\documentclass{amsart}

\usepackage{amsmath,amssymb}

\numberwithin{equation}{section}

\newtheorem{theorem}{Theorem}[section] 

\newtheorem{proposition}[theorem]{Proposition}
\newtheorem{corollary}[theorem]{Corollary}

\newcommand{\bigtimes}{\prod}

\begin{document}
\title[Menger algebras of $n$-place interior operations]{Menger algebras of $n$-place interior operations}
\author[W. A. Dudek]{Wieslaw A. Dudek}

\author[V. S. Trokhimenko]{Valentin S. Trokhimenko}

\begin{abstract}
Algebraic properties of $n$-place interior operations on a fixed
set are described. Conditions under which a Menger algebra of rank
$n$ can be represented by $n$-place interior operations are found.
\end{abstract}

\maketitle

\section{Introduction}
It is known \cite{Kur}, that on the topology on
a set $A$ one can talk in the language of open sets, or the
language of closed sets, or the language of interior operations
also called opening operations, or the language of closure
operations. Various types of closure operations on algebraic
systems and their applications are well described (see for example
\cite{Dudtro3}). So, a natural question is about a similar
characterization of interior operations having applications 
in topology and economics. Such operations from an
algebraic point of view were firstly studied by Vagner
\cite{Vagner}. Kulik observed in \cite{Kulik} that the
superposition of two interior operations is not always an interior
operation and found the conditions under which the composition of
two interior operations of a given set $A$ is also an interior
operation of this set. Moreover, he proved that a semigroup $S$ is
isomorphic to a semigroup of interior operations of some set if and
only if $S$ is idempotent and commutative.

Below we introduce the concept of $n$-place interior operations and
find conditions under which a Menger algebra of rank $n$ can be
isomorphically represented by $n$-place interior operations of some
set.

\section{Preliminaries} 
Let $A$ be a nonempty set, $\frak{P}(A)$---the
family of all subsets of $A$, $\mathcal{T}_n(\frak{P}(A))$---the
set of all $n$-place transformations of $\frak{P}(A)$, i.e., maps
$f\colon\stackrel{n}{\bigtimes}\frak{P}(A)\rightarrow\frak{P}(A)$, where
$\stackrel{n}{\bigtimes} \frak{P}(A)$ denotes the $n$-th
Cartesian power of the set $\frak {P}(A)$. For arbitrary
$f,g_1,\ldots,g_n\in \mathcal{T}_n(\frak{P}(A))$ we define the
$(n+1)$-ary composition $f[g_1\ldots g_n]$ by putting:
\[
f[g_1\ldots g_n](X_1,\ldots,X_n)=f(g_1(X_1,\ldots,X_n),\ldots,
g_n(X_1,\ldots,X_n))
\]
for all $X_1,\ldots,X_n\in\frak{P}(A)$.

The $(n+1)$-ary operation $\mathcal{O}\colon (f,g_1,\ldots,g_n)\mapsto f[g_1\ldots g_n]$ is called the
\textit{Menger superposition} of $n$-place functions (cf.\ 
\cite{Dudtro3, SchTr}). Then $(\mathcal{T}_n(\frak{P}(A)), \mathcal{O})$ 
is a \textit{Menger algebra} in the sense of \cite{Dudtro2} and \cite{Dudtro3}, i.e.,
the ope\-ration $\mathcal{O}$ satisfies the so-called
\textit{superassociative law}:
\begin{equation} \label{e1}
f[g_1\ldots g_n][h_1\ldots h_n]=f[g_1[h_1\ldots h_n]\ldots
g_n[h_1\ldots h_n]],
\end{equation}
where $f,g_i,h_i\in\mathcal{T}_n(\frak{P}(A))$, $i=1,\ldots,n$.

We say that $n$-place transformation $f\in\mathcal{T}_n(\frak{P}(A))$ is
\begin{itemize}
\item {\it contractive} if for any $X_1,\ldots,X_n\in\frak{P}(A)$
\begin{equation*}
  f(X_1,\ldots,X_n)\subseteq X_1\cap\dots\cap X_n;
\end{equation*}
\item {\it idempotent} if
\begin{equation*}
  f[f\ldots f]=f;
\end{equation*}
\item {\it isotone} if for any $X_1,\ldots,X_n,Y_1,\ldots,Y_n\in\frak{P}(A)$
\begin{equation*}
  X_1\subseteq Y_1\wedge\dots\wedge X_n\subseteq
  Y_n\Longrightarrow f(X_1,\ldots,X_n)\subseteq f(Y_1,\ldots,Y_n);
\end{equation*}
\item {\it $\cup$-distributive}, if
for all $X,Y,H_1,\ldots,H_n\in\frak{P}(A)$, $i=1,\ldots,n$
\[
 f(H_1^{i-1},X\cup Y,H_{i+1}^n)=f(H_1^{i-1},X,H_{i+1}^n)\cup
 f(H_1^{i-1},Y,H_{i+1}^n),
\]
where $H_s^r$ means $H_s,\ldots,H_r$ for $s\leqslant r$.
\item {\it $\cap$-distributive}, if
   for all $X,Y,H_1,\ldots,H_n\in\frak{P}(A)$, $i=1,\ldots,n$
\[
  f(H_1^{i-1},X\cap Y,H_{i+1}^n)=f(H_1^{i-1},X,H_{i+1}^n)\cap
  f(H_1^{i-1},Y,H_{i+1}^n);
\]
\item {\it full $\cup$-distributive}, if
   for any family of subsets $(X_k)_{k\in K}$ of $A$ and all $H_1,\ldots,H_n\in\frak{P}(A)$, $i=1,\ldots,n$
\[
 f\Big(H_1^{i-1},\bigcup_{k\in K}X_k,H_{i+1}^n\Big)=\bigcup_{k\in K}f\Big(H_1^{i-1},X_k,H_{i+1}^n\Big);
 \]
\item {\it full $\cap$-distributive}, if
   for any family of subsets $(X_k)_{k\in K}$ of $A$ and all $H_1,\ldots,H_n\in\frak{P}(A)$, $i=1,\ldots,n$
\[
 f\Big(H_1^{i-1},\bigcap_{k\in K}X_k,H_{i+1}^n\Big)=\bigcap_{k\in K}
 f\Big(H_1^{i-1},X_k,H_{i+1}^n\Big).
 \]
\end{itemize}

It is not difficult to see that the contractivity of an $n$-place
transformation $f\in\mathcal{T}_n(\frak{P}(A))$ is equivalent to
the system of conditions
\[
f(X_1,\ldots,X_n)\subseteq X_i,\quad i=1,\ldots,n .
\]
The isotonicity is equivalent to the system of $n$ implications
\[
  X\subseteq Y\Longrightarrow f(H_1^{i-1},X,H_{i+1}^n)\subseteq
  f(H_1^{i-1},Y,H_{i +1}^n), \quad i=1,\ldots,n,
\]
where $X,Y,X_1,\ldots,X_n,H_1,\ldots,H_n\in\frak{P}(A)$. 

It is also easy to show that the Menger superposition of contractive
(isotone) $n$-place transformations of $\frak{P}(A)$ is again a
contractive (isotone) $n$-place transformation of $\frak{P}(A)$.

An $n$-place transformation of $\frak{P}(A)$, which is
contractive, idempotent, and isotone is called an {\it $n$-place
interior operation} or an {\it $n$-place interior operator} on the set $A$.

For $n=1$ this definition coincides with the definition of interior
operations proposed by Vagner (see \cite{Vagner}).

\section{Properties of $n$-place interior operations}

We start with the following characterization of
$n$-place interior operations.

\begin{theorem} \label{T-1} For an $n$-place
transformation $f$ of $\frak{P}(A)$ the following conditions are
equivalent:
\begin{itemize}
\item[\textup{(a)}] $f$ is an $n$-place interior operation on $A$;
\item[\textup{(b)}] for all $X_1,\ldots,X_n,Y_1,\ldots,Y_n\in\frak{P}(A)$ we have
\begin{equation} \label{e2}
f(X_1\cap Y_1,\ldots,X_n\cap Y_n)\subseteq
f(f(X_1^n),\ldots,f(X_1^n))\cap f (Y_1^n)\cap Y_1\cap\dots\cap Y_n .
\end{equation}
\end{itemize}
\end{theorem}

\begin{proof} 
(a)~$\Longrightarrow$~(b). Suppose that $f$ is an $n$-place interior operation
on $A$. Then, by the contractivity of $f,$
for $X_i,Y_i\in\frak{P}(A)$, $i=1,\ldots,n$, we have
\begin{equation}\label{e3}
f(X_1\cap Y_1,\ldots,X_n\cap Y_n)\subseteq (X_1\cap
Y_1)\cap\dots\cap (X_n\cap Y_n)\subseteq Y_1\cap\dots\cap Y_n .
\end{equation}
As $X_i\cap Y_i\subseteq Y_i$, $i=1,\ldots,n$, the isotonity of $f$ implies
\begin{equation} \label{e4}
f(X_1\cap Y_1,\ldots,X_n\cap Y_n)\subseteq
f(Y_1,\ldots,Y_n)=f(Y_1^n).
\end{equation}
Similarly, $X_i\cap Y_i\subseteq X_i$, $i=1,\ldots,n$, implies
\begin{equation*}
f(X_1\cap Y_1,\ldots,X_n\cap Y_n)\subseteq
f(X_1,\ldots,X_n)=f(X_1^n).
\end{equation*}
Since $f(X_1^n)=f(f(X_1^n),\ldots,f(X_1^n))$, from the above we
obtain
\begin{equation*}
f(X_1\cap Y_1,\ldots,X_n\cap Y_n)\subseteq
f(f(X_1^n)\ldots,f(X_1^n)),
\end{equation*}
which together with  \eqref{e3} and \eqref{e4} gives \eqref{e2}.
Thus, (a) implies (b).

(b)~$\Longrightarrow$~(a).  If an $n$-place transformation $f$
satisfies \eqref{e2}, then putting in \eqref{e2} $X_i=Y_i$,
$i=1,\ldots,n$, we obtain
\begin{equation} \label{e5}
f(X_1^n)\subseteq f(f(X_1^n),\ldots,f(X_1^n))\cap f(X_1^n)\cap X_1\cap\dots\cap X_n.
\end{equation}
So, $f(X_1^n)\subseteq X_1\cap\dots\cap X_n$, i.e., $f$ is
contractive. In addition, \eqref{e5} implies $f(X_1^n)\subseteq
f(f(X_1^n),\ldots,f(X_1^n))$. Since $f$ is contractive, we have
\[
f(f(X_1^n),\ldots,f(X_1^n))\subseteq f(X_1^n)\cap\dots\cap f(X_1^n)=f(X_1^n),
\]
which together with the previous inclusion proves that $f$ is
idempotent.

If $X_i\subseteq Y_i$, then obviously $X_i\cap Y_i=X_i$. Hence, by
\eqref{e2}, for $X_i\subseteq Y_i$, $i=1,\ldots,n$, we have
\begin{align*}
f(X_1^n)&=f(X_1\cap Y_1,\ldots,X_n\cap Y_n) \\
&\subseteq(f(X_1^n),\ldots,f(X_1^n))\cap f(Y_1^n)\cap Y_1\cap\dots\cap Y_n\subseteq f(Y_1^n),
\end{align*}
which means that $f$ is isotone. Thus, $f$ is an $n$-place interior
operation. So, (b) implies (a).
\end{proof}

\begin{theorem} \label{T-2} For an $n$-place
transformation $f$ of  $\frak{P}(A) $ the following conditions are
equivalent:
\begin{itemize}
\item[\textup{(i)}] $f$ is contractive and full $\cup$-distributive;
\item[\textup{(ii)}] $f$ is contractive and  $\cup$-distributive;
\item[\textup{(iii)}] for all $X_1,\ldots,X_n\in\frak{P}(A)$ we have
\begin{equation} \label{e6}
 f(X_1^n)=f(A,\ldots,A)\cap X_1\cap\dots\cap X_n.
\end{equation}
\end{itemize}
\end{theorem}

\begin{proof}
(i)~$\Longrightarrow$~(ii).  Obvious.

(ii)~$\Longrightarrow$~(iii). According to the
$\cup$-distributivity, for all subsets $X,H_1,\ldots,H_n$ of $A$ 
and $i=1,\ldots,n$ we have
\[
f(H_1^{i-1},X,H_{i+1}^n)\cup f(H_1^{i-1},X',H_{i
+1}^n)=f(H_1^{i-1},A,H_{i+1}^n),
\]
where $X'=A\setminus X$. Then clearly
\begin{equation} \label{e7}
  \Big(f(H_1^{i-1},X,H_{i+1}^n)\cap X\Big)\cup\Big(f(H_1^{i-1},X',H_{i+1}^n)\cap X\Big)=  f(H_1^{i-1},A,H_{i+1}^n)\cap X.
\end{equation}

Since the map $f$ is contractive we have
$f(H_1^{i-1},X,H_{i+1}^n)\subseteq X$, which implies
$f(H_1^{i-1},X,H_{i+1}^n)\cap X=f(H_1^{i-1},X,H_{i+1}^n) $.
Similarly $f(H_1^{i-1},X',H_{i+1}^n)\subseteq X'$ and $X'\cap
X=\varnothing $ imply $f(H_1^{i-1},X',H_{i+1}^n)\cap
X=\varnothing$. Thus \eqref{e7} has the form
\[
f(H_1^{i-1},X,H_{i+1}^n)=f(H_1^{i-1},A,H_{i+1}^n)\cap X.
\]
Using this identity we obtain
\begin{align*}
& f(X_1,X_2,X_3,\ldots,X_n)=f(A,X_2,X_3,\ldots,X_n)\cap X_1 \\
& =f(A,A,X_3,\ldots,X_n)\cap X_2\cap X_1 = \cdots \\
& =f(A,A,A,\ldots,A)\cap X_n\cap\dots\cap X_2\cap X_1=f(A,\ldots,A)\cap X_1\cap\dots\cap X_n.
\end{align*}
So, (ii) implies (iii).

(iii)~$\Longrightarrow$~(i). It is not difficult to see that
\eqref{e6} implies $f(X_1,\ldots,X_n)\subseteq X_1\cap\dots\cap
X_n$. Thus, $f$ is contractive. Moreover, in this case we also
have
\begin{align*}
 & f(A,\ldots,A)\cap H_1\cap\dots\cap H_{i-1}\cap\Big(\bigcup_{k\in K}X_k\Big)
\cap H_{i+1}\cap\dots\cap H_n \\
&= \bigcup_{k\in K}\Big(f(A,\ldots,A)\cap H_1\cap\dots\cap
H_{i-1}\cap X_k\cap H_{i+1}\cap\dots\cap H_n\Big)\\
&= \bigcup_{k\in K} f\Big(H_1^{i-1},X_k,H_{i+1}^n\Big)
\end{align*}
by \eqref{e6}. 
So, $f$ is distributive with respect to the union. Thus (iii) implies (i), which completes the proof.
\end{proof}

\begin{corollary} \label{C-1} Every $n$-place transformation $f$ on $\frak{P}(A)$ satisfying $\eqref{e6}$ is an
$n$-place interior operation on $A$.
\end{corollary}
\begin{proof} Any transformation satisfying \eqref{e6} is clearly contractive. It is also idempotent because
\begin{align*}
 f(f(X_1^n),\ldots,f(X_1^n))&=f(A,\ldots,A)\cap f(X_1^n) \\
& =f(A,\ldots,A)\cap f(A,\ldots,A)\cap X_1\cap\dots\cap X_n  \\
& =f(A,\ldots,A)\cap X_1\cap\dots\cap X_n \\
&=f(X_1^n)
\end{align*}
for all $X_1,\ldots,X_n\in\frak{P}(A)$.

For $X_1\subseteq Y_1,\ldots,X_n\subseteq Y_n$ we have
$X_1\cap\dots\cap X_n\subseteq Y_1\cap\dots\cap Y_n$. Thus
$f(A,\ldots,A)\cap X_1\cap\dots\cap X_n\subseteq f(A,\ldots,A)
\cap Y_1\cap\dots\cap Y_n$. So, $f(X_1^n)\subseteq f(Y_1^n)$.
Hence $f$ is isotone.
\end{proof}

\begin{corollary} \label{C-2} 
Every \textup(full\textup) $\cup$-distributive
$n$-place interior operation is \textup(full\textup) $\cap$-distributive.
\end{corollary}
\begin{proof} Indeed, by Theorem \ref{T-2}, any $\cup$-distributive $n$-place interior
operation $f$ on the set $A$ satisfies \eqref{e6}. Hence
\begin{align*}
& f(H_1^{i-1},X\cap Y,H_{i+1}^n) \\
& =f(A,\ldots,A)\cap H_1\cap\dots\cap H_{i-1}\cap (X\cap Y)\cap H_{i+1}\cap\dots\cap H_n \\
& =\Big(f(A,\ldots,A)\cap H_1\cap\dots\cap H_{i-1}\cap
X\cap H_{i+1}\cap\dots\cap H_n\Big)\cap \\
& \hspace*{2cm}\Big(f(A,\ldots,A)\cap H_1\cap\dots\cap H_{i-1}\cap Y\cap
H_{i+1}\cap\dots\cap H_n\Big) \\
& =f(H_1^{i-1},X,H_{i+1}^n)\cap f(H_1^{i-1},Y,H_{i+1}^n)
\end{align*}
for $X,Y,H_1,\ldots,H_n\in\frak{P}(A)$ and $i=1,\ldots,n$. Thus
$f$ is $\cap$-distributive.

Analogously, we can show that $f$ is full $\cap$-distributive.
\end{proof}

\section{Compositions of $n$-place opening operations} 

On the set $\mathcal{T}_n(\frak{P}(A))$ of
$n$-place transformations of the set $A$ we introduce the binary
relation $\preceq$ defined by
\[
f\preceq g\;\Longleftrightarrow\;(\forall\,
X_1,\ldots,X_n)\,\big(\, f(X_1^n)\subseteq g(X_1^n)\,\big).
\]
It is easy to see that $\preceq$ is a partial order, i.e., it is
reflexive, transitive and antisymmetric.

\begin{proposition} \label{P-1} The relation $\preceq$ has the following properties:
\begin{enumerate}
\item[\textup{(a)}] If $f\in\mathcal{T}_n(\frak{P}(A))$ is contractive, then
\[
f[g_1\ldots g_n]\preceq g_i
\]
for all $g_1,\ldots,g_n\in\mathcal{T}_n(\frak{P}(A))$ and $i=1,\ldots,n$.
\item[\textup{(b)}] If $f\in\mathcal{T}_n(\frak{P}(A))$ is isotone, then
\[
g_1\preceq h_1\wedge\dots\wedge g_n\preceq h_n\Longrightarrow
f[g_1\ldots g_n]\preceq f[h_1\ldots h_n]
\]
for all
$g_1,\ldots,g_n,h_1,\ldots,h_n\in\mathcal{T}_n(\frak{P}(A))$.
\item[\textup{(c)}] If $f\in\mathcal{T}_n(\frak{P}(A))$ is isotone and $g\in\mathcal{T}_n(\frak{P}(A))$ is
contractive, then
\[
f[g\ldots g]\preceq f.
\]
\end{enumerate}
\end{proposition}
\begin{proof} If $f\in\mathcal{T}_n(\frak{P}(A))$ is contractive,
then for all $X_1,\ldots,X_n\in\frak{P}(A)$ and every
$i=1,\ldots,n$ we have
\[
f[g_1\ldots g_n](X_1^n)=f(g_1(X_1^n),\ldots,g_n(X_1^n))\subseteq
g_i(X_1^n).
\]
Hence, $f[g_1\ldots g_n]\preceq g_i$. So, the first property is
proved.

If $f\in\mathcal{T}_n(\frak{P}(A))$ is isotone and $g_i\preceq
h_i$ for $g_i,h_i\in\mathcal{T}_n(\frak{P}(A))$, $i=1,\ldots,n$,
then for all $X_1,\ldots,X_n\in\frak{P}(A)$ we have
$g_i(X_1^n)\subseteq h_i(X_1^n)$, $i=1,\ldots,n$, and consequently
$f(g_1(X_1^n),\ldots,g_n(X_1^n))\subseteq
f(h_1(X_1^n),\ldots,h_n(X_1^n))$. Therefore, $f[g_1\ldots
g_n](X_1^n)\subseteq f[h_1\ldots h_n](X_1^n)$. So, $f[g_1\ldots
g_n]\preceq f[h_1\ldots h_n]$, which proves the second condition.

The condition (c) is a consequence of (a) and (b).
\end{proof}

\begin{theorem} \label{T-3} The Menger superposition of given $n$-place interior
operations $f,g_1,\ldots,g_n$ defined on the set $A$ is an
$n$-place interior operation on $A$ if and only if for each
$i=1,\ldots,n$ we have
\begin{equation} \label{e8}
g_i[f\ldots f][g_1\ldots g_n]=f[g_1\ldots g_n].
\end{equation}
\end{theorem}
\begin{proof} $(\Longrightarrow)$.  Suppose that $f,g_1,\ldots,g_n$ and $f[g_1\ldots g_n]$ are $n$-place
interior ope\-ra\-tions on $A$. Since each $g_i$ is contractive,
according to \eqref{e1} and Proposition \ref{P-1}, we obtain
\[
g_i[f\ldots f][g_1\ldots g_n]=g_i[f[g_1\ldots g_n]\ldots
f[g_1\ldots g_n]]\preceq f[g_1\ldots g_n].
\]
On the other side, using \eqref{e1} and the fact that $f[g_1\ldots
g_n]$ is an idempotent $n$-place transformation, we get
\begin{align*}
 f[g_1\ldots g_n]&=f[g_1\ldots g_n][f[g_1\ldots g_n]\ldots f[g_1\ldots g_n]] \\
&= f[g_1[f[g_1\ldots g_n]\ldots f[g_1\ldots g_n]]\ldots
g_n[f[g_1\ldots g_n]\ldots f[g_1\ldots g_n]]]\\
& \preceq g_i[f[g_1\ldots g_n]\ldots f[g_1\ldots g_n]]\\
&=g_i[f\ldots f][g_1\ldots g_n] ,
\end{align*}
which together with the previous inequality gives \eqref{e8}.

$(\Longleftarrow)$.  As it was mentioned above, Menger
superposition preserves contractivity and isotonicity. We show
that $f[g_1\ldots g_n]$ is idempotent. Indeed, according to
\eqref{e1} and \eqref{e8}, we have

\begin{align*}
& f[g_1\ldots g_n][f[g_1\ldots g_n]\ldots f[g_1\ldots g_n]] \\
& =f[g_1[f[g_1\ldots g_n]\ldots f[g_1\ldots g_n]]\ldots
g_n[f[g_1\ldots g_n]\ldots f[g_1\ldots g_n]]] \\
& =f[g_1[f\ldots f][g_1\ldots g_n]\ldots g_n[f\ldots f][g_1\ldots g_n]] \\
& =f[f[g_1\ldots g_n]\ldots f[g_1\ldots g_n]]= f[f\ldots f]
[g_1\ldots g_n]\\
& =f[g_1\ldots g_n].
\end{align*}
Thus, $f[g_1\ldots g_n]$ is an $n$-place interior operation.
\end{proof}

\section{Algebras derived from their diagonal semigroups}

Recall (cf.\ \cite{Dudtro2}, \cite{Dudtro3}),
that a {\it Menger algebra $(G,o)$ of rank $n$} is a nonempty set
$G$ with an $(n+1)$-ary operation $o\colon(f,g_1,\ldots,g_n)\mapsto
f[g_1\ldots g_n]$ satisfying the identity \eqref{e1}. On such
algebra we can define a binary operation $*$ by setting
$$x*y=x[y\ldots y]$$ for any $x,y\in G$. It is easy to see that
$(G,*)$ is a semigroup. It is called the {\it diagonal semigroup}
of $(G,o)$. In the case when
$$
f[g_1g_2\ldots g_n]=f*g_1*g_2*\ldots *g_n
$$
we say that a Menger algebra $(G,o)$ is {\it derived} from its
diagonal semigroup $(G,*)$.

\begin{proposition} \label{P-2} For any $n$-place interior operations $f,g$ on the set $A$ the following
three conditions are equivalent:
\begin{enumerate}
\item[\textup{(a)}] \ $f*g$ is an $n$-place interior operation on $A$;
\item[\textup{(b)}] \ $f*g=g*f*g$;
\item[\textup{(c)}] \ $f*g=f*g*f$.
  \end{enumerate}
\end{proposition}
\begin{proof} (a)~$\Longrightarrow$~(b).  This implication follows from Theorem \ref{T-3}.

(b)~$\Longrightarrow$~(c). By Proposition \ref{P-1} we
have $f[g\ldots g]\preceq f$, i.e., $f*g\preceq f$. This together
with (b) shows $f*g=g*f*g\preceq g*f$. So, $f*g\preceq
g*f$, which in view of Proposition \ref{P-1} (b) gives
$f*f*g\preceq f*g*f$. Hence, $f*g=f*f*g\preceq f*g*f\preceq f*g$,
and consequently, $f*g=f*g*f$.

(c)~$\Longrightarrow$~(a).  Since Menger
superposition preserves contractivity and isotonicity, then $f*g$
is an isotone and contractive $n$-place transformation. We show
that it is idempotent. In fact,
\[
(f*g)*(f*g)=(f*g*f)*g=(f*g)*g=f*(g*g)=f*g.
\]
Thus, $f*g$ is an $n$-place interior operation.
\end{proof}

\section{Characterizations of algebras of $n$-place interior operations} 

An abstract characterization of Menger algebras
of $n$-place interior operations is given in the following theorem.

\begin{theorem} \label{T-4} A Menger algebra $(G,o)$
of rank $n$ is isomorphic to a Menger algebra of $n$-place
interior operations on some set if and only if it satisfies the
following three identities
\begin{align}
& \label{e9} x[x\ldots x]=x, \\
& \label{e10} x[y\ldots y]=y[x\ldots x], \\
& \label{e11} x[y_1\ldots y_n]=x[y_1\ldots y_1]\ldots [y_n\ldots y_n].
\end{align}
\end{theorem}
\begin{proof} {\sc Necessity.} Let $(\Phi,\mathcal{O})$, where $\Phi\subset\mathcal{T}_n(\frak{P}(A))$, be a
Menger algebra of $n$-place interior operations on the set $A$.
Then obviously $f[f\ldots f]=f$ for all $f\in\Phi$. Thus, the
condition \eqref{e9} is satisfied.

If $f,g\in\Phi$, then also $f[g\ldots g], g[f\ldots f]\in\Phi$.
Therefore $f*g$ and $g*f$ are $n$-place interior operations and by
(b), (c) from Proposition \ref{P-2} we have
$f*g=g*f$. Thus, $f[g\ldots g]=g[f\ldots f]$. Hence the condition
\eqref{e10} also is satisfied.

Further, using \eqref{e1} and Theorem \ref{T-3}, for
$f,g_1,\ldots,g_n\in\Phi$ and for each $i=1,\ldots,n$, we obtain
\[
g_i*f[g_1\!\ldots g_n]=g_i[f[g_1\!\ldots g_n]\ldots f[g_1\!\ldots g_n]]=
g_i[f\ldots f][g_1\!\ldots g_n]=\!f[g_1\!\ldots g_n].
\]
Thus
\[
  g_i*f[g_1\ldots g_n]=f[g_1\ldots g_n]
\]
for all $i=1,\ldots,n$. Consequently
\begin{align*}
 (f*g_1*\ldots*g_i)*f[g_1\ldots g_n]&
= (f*g_1*\ldots*g_{i-1})*(g_i*f[g_1\ldots g_n]) \\
&=(f*g_1*\ldots*g_{i-1})*f[g_1\ldots g_n]\\
&\hspace*{6pt}\vdots \\
&=f*f[g_1\ldots g_n] \\
& =f[f[g_1\ldots g_n]\ldots f[g_1\ldots g_n]]\\
&= f[f\ldots f][g_1\ldots g_n]\\
&=f[g_1\ldots g_n].
\end{align*}
Hence
\begin{equation}\label{e11a}
(f*g_1*\ldots *g_i)*f[g_1\ldots g_n]=f[g_1\ldots g_n]
\end{equation}
for every $i=1,\ldots,n$. Since $f[g_1\ldots g_n]\preceq g_n$ and
\[
(f*g_1*\ldots*g_{n-1})*f[g_1\ldots g_n]\preceq (f*g_1*\ldots
*g_{n-1})*g_n,
\]
by Proposition \ref{P-1}, the equality \eqref{e11a} means that
\begin{equation} \label{e12}
f[g_1\ldots g_n]\preceq f* g_1*\ldots*g_n.
\end{equation}
By Proposition \ref{P-1} we also have
\[
f*g_1*\ldots*g_n\preceq f*g_1*\ldots*g_i\preceq g_i
\]
for all $i=1,\ldots,n$. Therefore
\[
f[(f*g_1*\ldots*g_n)\ldots (f*g_1*\ldots*g_n)]\preceq f[g_1\ldots
g_n],
\]
i.e., $f*(f*g_1*\ldots*g_n)\preceq f[g_1\ldots g_n]$, whence, by
$f*f=f$, we obtain
\[
  f*g_1*\ldots *g_n\preceq f[g_1\ldots g_n].
\]
This, together with \eqref{e12}, gives $f[g_1\ldots
g_n]=f*g_1*\ldots*g_n$. Thus, the condition \eqref{e11} is
satisfied too.

{\sc Sufficiency.} Let $(G,o)$ be a Menger algebra of rank $n$
satisfying all the conditions of the theorem and let $(G,*)$ be
its diagonal semigroup. Consider a binary relation $\omega$
defined on the set $G$ as follows:
\[
  \omega=\{(x,y)\mid x*y=y\}.
\]
Since the diagonal semigroup $(G,*)$ is semilattice, the relation $\omega $ is reflexive, transitive and
antisymmetric. So, the relation $\omega$ is just the usual ordering on a semilattice (treated as a join semilattice), and  
\[
\omega\langle x\rangle=\{y\in G\mid (x,y)\in\omega\}
\]
is just the upset in this ordering.

We show that
\begin{equation}\label{e13}
  \omega\langle x[y_1\ldots y_n]\rangle=\omega\langle x\rangle\cap\omega\langle y_1\rangle
\cap\dots\cap\omega\langle y_n\rangle
\end{equation}
for $x,y_1,\ldots,y_n\in G$.

Since the diagonal semigroup $(G,*)$ is semilattice, the following equation
\[
\omega\langle x*y\rangle=\omega\langle x\rangle\cap\omega\langle
y\rangle
\]
holds in $(G,*)$.

Using this equation and \eqref{e11} we obtain
\begin{align*}
\omega\langle x[y_1\ldots y_n]\rangle&=\omega\langle x*y_1
*\dots *y_n\rangle=\omega\langle x*(y_1*\dots *y_n)\rangle\\
&=\omega\langle x\rangle\cap\omega\langle y_1\!*(y_2\!*\dots
*y_n)\rangle =\omega\langle x\rangle\cap\omega\langle
y_1\rangle\cap\omega\langle y_2*\dots*y_n\rangle \\
&= \dots=\omega\langle x\rangle\cap\omega\langle y_1\rangle
\cap\dots\cap\omega\langle y_n\rangle,
\end{align*}
which proves \eqref{e13}.

Consider the set $\Phi=\{f_g\mid g\in G\}\subseteq\mathcal{T}_n
(\frak{P}(G))$ of all $n$-place operations $f_g$ defined by:
\begin{equation} \label{e14}
f_g(X_1^n)=\omega\langle g\rangle\cap X_1\cap\dots\cap X_n
\end{equation}
and the map $P\colon g\mapsto f_g$.

Obviously $f_g(G,\ldots,G)=\omega\langle g\rangle$. Thus
$f_g(X_1^n)=f_g(G,\ldots,G)\cap X_1\cap\dots\cap X_n$. Hence, by
Corollary \ref{C-1}, $f_g$ is an $n$-place interior operation on
the set $G$. More\-over, for all $g,g_1,\ldots,g_n$ and
$X_1,\ldots,X_n\in\frak{P}(G)$ we have
\begin{align*}
 &f_g [f_{g_1}\ldots f_{g_n}] (X_1^n)=f_g(f_{g_1}(X_1^n),
\ldots,f_{g_n}(X_1^n)) \\
&= \omega\langle g\rangle\cap f_{g_1}(X_1^n)\cap\dots\cap f_{g_n}(X_1^n) \\
&= \omega\langle g\rangle\cap\big(\omega\langle g_1\rangle\cap
X_1\cap\dots\cap X_n\big)\cap\dots\cap\big(\omega\langle
g_n\rangle\cap X_1\cap\dots\cap X_n\big)\\
& =\big(\omega\langle g\rangle\cap\omega\langle g_1
\rangle\cap\dots\cap\omega\langle g_n\rangle\big)\cap X_1
\cap\dots\cap X_n\\
&\stackrel{\eqref{e13}}{=} \omega\langle g[g_1\ldots g_n]\rangle\cap X_1\cap\dots\cap
X_n=f_{g[g_1\ldots g_n]}(X_1^n).
\end{align*}
Thus, $P(g[g_1\ldots g_n])=P(g)[P(g_1)\ldots P(g_n)]$, i.e., $P$
is a homomorphism of $G$ onto $\Phi$.

Suppose now that $P(g_1)=P(g_2)$, where $g_1,g_2\in G$. Then
$f_{g_1}=f_{g_2}$, i.e., $f_{g_1}(X_1^n)=f_{g_2}(X_1^n)$ for all
$X_1,\ldots,X_n\in\frak{P}(G)$. Hence $\omega\langle g_1
\rangle\cap X_1\cap\dots\cap X_n=\omega\langle g_2\rangle\cap
X_1\cap\dots\cap X_n$, which for $X_1=\dots=X_n=G$ gives
$\omega\langle g_1\rangle=\omega\langle g_2\rangle$. Since
$\omega$ is an antisymmetric relation, from the above we conclude
$g_1=g_2$. Thus $P$ is a bijection of $G$ onto $\Phi$. This means
that a Menger algebra $(G,o)$ of rank $n$ is isomorphic to the
constructed Menger algebra $(\Phi,\mathcal{O})$ of $n$-place
interior operations.
\end{proof}

\begin{corollary}\label{C-3}
A Menger algebra $(G,o)$ of rank $n$ is isomorphic to a Menger
algebra of $n$-place interior operations on some set if and only if
it is derived from an idempotent commutative semigroup.
\end{corollary}

The above corollary says that a Menger algebra isomorphic to a
Menger algebra of $n$-place interior operations is an idempotent
commutative $n$-ary semigroup. Since a diagonal semigroup of a
group-like Menger algebra is a group (see \cite{Du'86} or
\cite{Dudtro3}), a group-like Menger algebra isomorphic to a
Menger algebra of $n$-place interior operations has only one
element. In view of Corollary \ref{C-3}, it is obvious that two Menger algebras of $n$-place interior
operations are isomorphic if and only if their diagonal semigroups
are isomorphic.

According to Theorem \ref{T-2} and Corollary \ref{C-2} each
$n$-place interior operation defined by \eqref{e14} is (full)
$\cup$-distributive and (full) $\cap$-distributive. Therefore, we
have the following corollary.

\begin{corollary} \label{C-4} For a Menger algebra $(G,o)$ of rank $n$ the following conditions are
equivalent:
\begin{itemize}
\item $(G,o)$ is derived from an idempotent commutative semigroup;
\item $(G,o)$ is isomorphic to a Menger algebra of contractive \textup(full\textup) $\cup$-distributive $n$-place
transformations on some set;
\item $(G,o)$ is isomorphic to a Menger algebra of \textup(full\textup) $\cup$-distributive $n$-place interior operations
on some set.
\item $(G,o)$ is isomorphic to a Menger algebra of \textup(full\textup) $\cup$-distributive and $\cap$-distributive
$n$-place interior operations on some set.
\end{itemize}
\end{corollary}

\section{Semigroups of interior operations} 

Menger algebras of rank $n=1$ are (binary)
semigroups. So, as a consequence of our results, we obtain some
useful facts for semigroups.

Recall that Vagner (cf.\ \cite{Vagner}) defined \textit{interior
operations} on the set $A$ as contractive, idempotent and isotone
transformations of $\frak{P}(A)$. This definition coincides with
our definition for $n=1$. So, as a consequence of Theorem
\ref{T-1} we obtain

\begin{corollary} \label{C-5} A transformation $f$ of \ $\frak{P}(A)$ is an interior operation
on the set $A$ if and only if
\[
f(X\cap Y)\subseteq f(f(X))\cap f(Y)\cap Y
\]
is valid for all $X,Y\in\frak{P}(A)$.
\end{corollary}

From Theorem \ref{T-2} and Corollary \ref{C-1} we obtain

\begin{corollary}\label{C-6} For a transformation $f$ of \ $\frak{P}(A)$ the following conditions are equivalent:
\begin{itemize}
\item  $f$ is contractive and full $\cup$-distributive;
\item  $f$ is contractive and $\cup$-distributive;
\item  for every $X\in\frak{P}(A)$ we have
\begin{equation} \label{e15}
 f(X)=f(A)\cap X.
\end{equation}
\end{itemize}
\end{corollary}

\begin{corollary} \label{C-7} Any transformation $f$ of \
$\frak{P}(A)$ satisfying $\eqref{e15}$ is an interior operation on
$A$.
\end{corollary}

As a consequence of Corollary \ref{C-2} we obtain

\begin{corollary} \label{C-8} Every
\textup(full\textup) $\cup$-distributive interior operation is \textup(full\textup)
$\cap$-dis\-tri\-butive.
\end{corollary}

Putting $n=1$ in Theorem \ref{T-3} we obtain one of the main
results of the paper \cite{Kulik}:

\begin{corollary}\label{C-9} The composition $f\circ g$ of two interior operations defined on the same set is
an interior operation on this set if and only if \ $f\circ g=f\circ g\circ f$.
\end{corollary}

The other main result of \cite{Kulik} is a consequence of our
Corollary \ref{C-4} which for $n=1$ has the form:

\begin{corollary} \label{C-10}
For a semigroup $(G,\cdot)$ the following statements are
equivalent:
\begin{itemize}
\item $(G,\cdot)$ is an idempotent commutative semigroup;
\item $(G,\cdot)$ is isomorphic to a semigroup of contractive and \textup(full\textup) $\cup$-distributive
transformations on some set;
\item $(G,\cdot)$ is isomorphic to a semigroup of \textup(full\textup) $\cup$-distributive interior operations
on some set. 
\item $(G,\cdot)$ is isomorphic to a semigroup of
\textup(full\textup) $\cup$-distributive and $\cap$-distri\-butive interior
operations on some set.
\end{itemize}
\end{corollary}

From this we obtain

\begin{corollary}[Podluzhnyak and Kulik~\cite{Kulik}]
\label{C-11}
A semigroup $(G,\cdot)$ is isomorphic to some semigroup of interior
operations on some set $A$ if and only if it is idempotent and
commutative.
\end{corollary}

\end{document}